\documentclass[11pt,a4paper,oneside,reqno]{amsart}

\usepackage{amsmath, amsthm, amsopn, amssymb}
\usepackage{mathtools, thmtools, thm-restate}
\usepackage{bbm}
\usepackage{comment}
\usepackage{enumerate}
\usepackage{enumitem}
\usepackage{hyperref}
\usepackage[abbrev,msc-links,backrefs]{amsrefs}
\usepackage{cleveref}
\usepackage{tikz}

\usepackage{stmaryrd}
\usepackage{xcolor}

\hypersetup{colorlinks=true,citecolor=blue,linktoc=page}
\hypersetup{%
  colorlinks,
  linkcolor={red!50!black},
  citecolor={green!50!black},
  urlcolor={blue!50!black}
}

\usepackage{geometry}
\usepackage{setspace}
\newtheorem{theorem}{Theorem}[section]
\newtheorem{lemma}[theorem]{Lemma}
\newtheorem{proposition}[theorem]{Proposition}
\newtheorem{claim}[theorem]{Claim}

\newtheorem{conjecture}[theorem]{Conjecture}

\newtheorem{question}[theorem]{Question}

\newtheorem{definition}[theorem]{Definition}

\newcommand{\RR}{\mathbb{R}}
\newcommand{\PP}{\mathbb{P}}
\newcommand{\EE}{\mathbb{E}}

\newcommand{\cA}{\mathcal{A}}

\newcommand{\cC}{\mathcal{C}}

\newcommand{\cT}{\mathcal{T}}

\DeclareMathOperator{\Bin}{Bin}
\DeclareMathOperator{\dist}{dist}

\DeclareMathOperator{\Var}{Var}
\newcommand{\eps}{\varepsilon}

\def\1{\mathbbm{1}}

\author{Asaf Ferber}
\address{Mathematics Department, University of California Irvine, Irvine 92697, USA}
\email{asaff@uci.edu}

\author{Michael Krivelevich}
\address{School of Mathematical Sciences, Tel Aviv University, Tel Aviv 6997801, Israel}
\email{krivelev@tauex.tau.ac.il}

\author{Marcelo Sales}
\address{Mathematics Department, University of California Irvine, Irvine 92697, USA}
\email{mtsales@uci.edu}

\author{Wojciech Samotij}
\address{School of Mathematical Sciences, Tel Aviv University, Tel Aviv 6997801, Israel}
\email{samotij@tauex.tau.ac.il}

\thanks{AF and MS are supported by US Air force grant FA9550-23-1-0298. MK is partially supported by NSF-BSF grant 2023688.  WS is supported by the Israel Science Foundation grant 2110/22 and the ERC Consolidator Grant 101044123 (RandomHypGra).  Finally, AF and WS were also supported by the grant 2019679 from the United States National Science Foundation (NSF) and the United States--Israel Binational Science Foundation (BSF)}

\title[On the edge expansion of random polytopes]{On the edge expansion of random polytopes}

\begin{document}

\begin{abstract}
  A \emph{$0/1$-polytope} in $\RR^n$ is the convex hull of a subset of $\{0,1\}^n$. The graph of a~polytope $P$ is the graph whose vertices are the zero-dimensional faces of $P$ and whose edges are the one-dimensional faces of~$P$.  A conjecture of Mihail and Vazirani states that the edge expansion of the graph of every $0/1$-polytope is at least one.  We study a random version of the problem, where the polytope is generated by selecting vertices of $\{0,1\}^n$ independently at random with probability $p\in (0,1)$.  Improving earlier results, we show that, for any $p\in (0,1)$, with high probability the edge expansion of the random $0/1$-polytope is bounded from below by an absolute constant.
\end{abstract}

\maketitle

\section{Introduction}\label{sec:intro}

A \emph{$0/1$-polytope} in $\RR^n$ is the convex hull of a subset of $\{0,1\}^n$, i.e., a polytope whose vertices have all coordinates either $0$ or $1$. These polytopes are the central object of study in polyhedral combinatorics, due to their connections to linear programming and combinatorial optimization. Most of these connections arise from the ability to encode combinatorial objects via characteristic vectors. To be more precise, given a set system $\cA$ with a ground set of size~$n$, one can consider the associated $0/1$-polytope, which is the convex hull of the characteristic vectors of all the elements of~$\cA$. For many combinatorial objects (e.g., matchings, matroids, order ideals, independent sets), interesting structural properties can be expressed as geometric properties of the associated polytopes.

For a polytope $P$, the graph $G_P$ of $P$ is the graph whose vertices are the zero-dimensional faces of $P$ and whose edges are the one-dimensional faces of $P$. Several properties of the graph of $0/1$-polytopes have been studied in the past \cites{combpolytope06, graphpolytope84, graphpolytopeII84, graphpolytope89, cutpolytope}. Here we focus on their expansion.
For a graph $G$ with vertex set $V$, we define the \emph{edge expansion} of $G$ (also known as the \emph{Cheeger constant} of $G$) by
\[
  h(G) \coloneqq \min\left\{\frac{e(S,V\setminus S)}{|S|} : S \subseteq V \text{ and } 1 \leq |S| \leq \frac{|V|}{2}\right\}.
\]

It is well known (see, e.g., \cites{expansionpolysurvey92, markovchain96}) that if $G$ is the graph of a $0/1$-polytope whose vertices have degree bounded by a polynomial in $n$, then a lower bound on the Cheeger constant of $G$ translates to an upper bound on the mixing time of a random walk on $G$.
Performing random walks on the graphs of $0/1$-polytopes can be used to uniformly generate random elements in classes of combinatorial objects.
In many cases, this allows us to design randomized algorithms that approximately count the number of objects; the running time of such algorithms is inversely proportional to the Cheeger constant of the graph of the underlying $0/1$-polytope.
For instance, this method was used by Jerrum, Sinclair, and Vigoda \cites{mixingapplication89, mixingapplication04} to design a polynomial-time approximation algorithm for computing the permanent of a matrix with nonnegative entries. 
Inspired by these connections, Mihail and Vazirani \cites{expansionpolysurvey92, mihailvaziraniconj92} made the following conjecture about the edge expansion of the graph of an arbitrary $0/1$-polytope.
A proof of this conjecture would have many important applications to the analysis of randomized algorithms \cites{mihailvaziraniconj92, gillmannthesis07, expansionpolytope04}.

\begin{conjecture}\label{conj:mihailvazirani}
  Every $0/1$-polytope $P$ satisfies $h(G_P) \geq 1$.
\end{conjecture}

In other words, \Cref{conj:mihailvazirani} states that the Cheeger constant of every $0/1$-polytope is at least as large as that of the hypercube, as it is well known~\cites{Harper64 ,B67, Hart76, Lindsey64} that $h(Q^n) = 1$ for all $n$, see \Cref{thm:Harper} below.  The conjecture has been verified for a variety of polytopes associated with combinatorial objects, such as perfect matching polytopes, order ideal polytopes, and matroid polytopes \cites{expansionpolytope04, mihailthesis89}.  In a recent breakthrough, Anari, Liu, Gharam, and Vinzant \cite{expansionpolytope24} showed that the conjecture holds for matroid base polytopes, i.e., any $0/1$-polytope associated with a matroid. Despite this progress, \Cref{conj:mihailvazirani} remains wide open for general $0/1$-polytopes.

Another special class of polytopes, and perhaps an interesting intermediate step in this context, is that of random polytopes.  Let $Q^n \coloneqq \{0,1\}^n$ be the $n$-dimensional hypercube.  We consider the following model of random $0/1$-polytopes: Given $p \in (0,1)$, let $U$ be a random subset of $Q^n$, where each element is selected independently with probability $p$. We define the random polytope $P_{n,p}$ to be the convex hull of $U\subseteq \{0,1\}^n$.  The problem of estimating the expansion of $P_{n,p}$ was introduced by Gillmann \cite{gillmannthesis07}. In recent work, Leroux and Rademacher \cite{randomexppolytopes24} showed for every $p\in (0,1)$ that, with high probability\footnote{An event $A$, or rather a sequence $(A_n)$ of events indexed by the dimension $n$, happens \emph{with high probability} (\emph{whp} for short) if the probability of $A_n$ tends to one as $n$ tends to infinity.}, the graph of the polytope $P_{n,p}$ has expansion at least $1/(12n)$. Our main result is an improvement of this bound to a constant.

\begin{theorem}\label{thm:main}
  There exist absolute constants $\beta, \eta>0$ such that the following holds for sufficiently large~$n$.  If $p = p(n)\geq 2^{-0.99n}$, then whp the graph $G$ of $P_{n,p}$ satisfies $h(G)\geq \beta$;  moreover, for every $A \subseteq V(G)$ with $|A|\leq \eta |V(G)|$, we have $e(A, V(G) \setminus A)\geq |A|$.
\end{theorem}

Note that \Cref{thm:main} does not apply when the density \( p \) is very small. To complement this, we use a result of Bondarenko and Brodskii~\cite{BondarenkoBrodskiiCliques}, which states that for \( p \leq 2^{-5n/6} \), the polytope graph is, with high probability, a clique. For the sake of completeness, we include a short proof of a weaker version of this result.

\begin{proposition}\label{prop:sparse}
  For any $\varepsilon >0$, if $p = p(n)\leq c^{n}$ for $c<7^{-1/3}$, then whp the graph $G$ of $P_{n,p}$ is complete and thus $h(G) \ge |P_{n,p}| - 1$.
\end{proposition}

\section{Preliminaries}\label{sec:preliminaries}

\subsection{Proof overview}\label{subsec:outline}

In this section, we describe our proof strategy for \Cref{thm:main}. For the sake of clarity and comparison purposes, we will first briefly describe the approach developed in \cite{randomexppolytopes24}. Given $d<n$, let $\pi \colon Q^n \rightarrow Q^d$ be the projection onto the first $d$ coordinates.  Note that the projection $\pi$ naturally partitions $Q^n$ into $2^d$ disjoint preimages of size $2^{k}$, where $k \coloneqq n-d$.  This fact implies that, for sufficiently large  $k$, we have whp that the random polytope $P:=P_{n,p} \subseteq Q^n$ projects in a balanced way onto the full hypercube, i.e., $\pi(P) = Q^d$, and the size of each \emph{fiber} $P \cap \pi^{-1}(\{x\})$ concentrates around its mean $p2^k$, simultaneously for all $x \in Q^d$.

Given a set $A \subseteq P$, one can classify the fibers of all vertices of $Q^d$ into three types: the subset $U \subseteq Q^d$ of those fully occupied by elements of $A$, the subset $M \subseteq Q^d$ of those partially occupied by elements of $A$, and the fibers that are disjoint from $A$. The heart of \cite{randomexppolytopes24} is a projection lemma that ensures that either there are many edges in $E_{G_P}(A, A^\mathsf{c})$ coming from the fibers in $M$, or there are many edges from $U$ to $U^\mathsf{c}$. The latter relies on the fact that the set $U$ has good expansion in $Q^d$ and that $\pi(P) = Q^d$. Unfortunately, the resulting bound on the expansion of $G_P$ is inversely proportional to $p2^k$, which is chosen to be $\Omega(d)$ in order to ensure that $\pi(P) = Q^d$.

In the proof of \Cref{thm:main}, to achieve constant edge expansion, we select $k \coloneqq n-d$ such that $p2^k$ is an absolute constant.  This choice introduces several challenges whose overcoming requires new ideas. First, since the projection $R \coloneqq \pi(P)$ is no longer the full hypercube, the projection lemma from \cite{randomexppolytopes24} can no longer be used.  To overcome this, we develop a more general projection lemma (see \Cref{lem:projection}) that allows us to obtain a lower bound on $e_{G_P}(A, A^\mathsf{c})$ in terms of the edge expansion of the projected polytope $R$.  Second, since $p2^k$ is a constant, one cannot ensure that the projection is balanced, requiring a much more careful analysis of the typical projection of $P$ (see \Cref{subsec:randprojection}). Finally, since we cannot rely on the hypercube's strong expansion properties, we establish a new edge-isoperimetric inequality tailored for very dense random $0/1$-polytopes (see \Cref{thm:verydense}).  This last result is perhaps the main technical contribution of our work.

The paper is organized as follows. In the remainder of \Cref{sec:preliminaries}, we introduce the probabilistic and geometric tools used throughout the paper. In particular, \Cref{subsec:projection} contains the proof of our version of the projection lemma.  The short \Cref{sec:sparse} is dedicated to the proof of \Cref{prop:sparse}.  In \Cref{sec:verydense}, we establish an edge-isoperimetric inequality for very dense polytopes, while \Cref{sec:general} presents the proof of \Cref{thm:main}.

\subsection{Concentration}

We shall use the following well-known estimate for tail probabilities of binomial random variables (see, e.g., \cite[Theorem~2.1]{JRLrandomgraphs}).

\begin{theorem}\label{chernoff}
  For every positive integer $n$ and all $p \in [0,1]$ and all $\alpha \ge 0$,
  \[
    \PP\bigl(\Bin(n,p) \ge (1+\alpha) n p\bigr) \le \exp\bigl(-np \cdot ((1+\alpha) \log (1+\alpha) - \alpha)\bigr) \le \exp\left(-\frac{\alpha^2np}{2(1+\alpha)}\right).
  \]
  In particular, for every $\ell \ge enp$,
  \[
    \PP\bigl(\Bin(n,p) \ge \ell\bigr) \le \exp\left( -\ell\cdot\log\left(\frac{\ell}{enp}\right)\right).
  \]
\end{theorem}

In addition, we shall use the following tail estimate for the distance of the cumulative distribution function of a random variable $Y$ from its empirical counterpart determined by a sequence of independent copies of $Y$ due to Dvoretzky, Kiefer, and Wolfowitz~\cite{DvoKieWol56}.
\begin{theorem}[DKW Inequality]
  \label{thm:DKW}
  Suppose that $Y_1, \dotsc, Y_n$ is a sequence of i.i.d.\ real-valued random variables and let $D_n \colon \mathbb{R} \to [0,1]$ be the associated empirical distribution function, i.e.,
  \[
    D_n(y) \coloneqq \frac{1}{n} \sum_{i =1}^n \1_{Y_i \le y}.
  \]
  Then, for every $\eps > 0$,
  \[
    \PP\left(\sup_{y \in \mathbb{R}} \bigl|D_n(y) - \PP(Y_1 \le y)\bigr| > \eps\right) \le 2\exp(-2\eps^2n).
  \]
\end{theorem}

Finally, we will need the following technical estimate that quantifies the fact that the binomial distribution is tightly concentrated around its mean.

\begin{lemma}
  \label{lemma:binomial-tilt}
  There exists an absolute constant $t_0$ such that the following holds.
  Suppose that a positive integer $n$ and $p \in [0,1]$ satisfy $t \coloneqq np \ge t_0$ and let $T \sim \Bin(n,p)$.
  Then, for every integer $m$ such that $\PP(T \le m+1) \ge 3/5$, we have $\mathbb{E}[T \cdot \1_{T \le m}] \ge 5t/9$.
\end{lemma}
\begin{proof}
  Let $m$ be an integer satisfying $\PP(T \le m+1) \ge 3/5$ and observe that
  \[
    t - \mathbb{E}[T \cdot \1_{T \le m}] = \mathbb{E}[T \cdot \1_{T > m}] \le t \cdot \left(\PP(T > m) + \int_{0}^{\infty} \PP\bigl(T \ge (1+x)t\bigr)\, dx\right).
  \]
  Since, for some absolute constant $C$,
  \[
    \PP(T > m) = \PP(T > m+1) + \PP(T = m+1) \le 2/5 + C/\sqrt{t}
  \]
  and, by~\Cref{chernoff},
  \[
    \int_{0}^{\infty} \PP\bigl(T \ge (1+x)t\bigr)\, dx \le \int_{0}^{\infty} \exp\left(-\frac{x^2t}{2(1+x)}\right) \, dx \le \frac{C}{\sqrt{t}},
  \]
  we may conclude that $\mathbb{E}[T \cdot \1_{T \le m}] \ge 5t/9$ whenever $t \ge t_0$ for sufficiently large constant $t_0$.
\end{proof}

\subsection{Geometry of polytopes}\label{subsec:geometry}

We now introduce our notation for polytopes and present several geometric results that will be useful throughout the paper. We refer the reader to \cites{lecturepolytopes, lectures01polytopes} for a comprehensive introduction to convex polytopes and $0/1$-polytopes.

A subset $P\subseteq \RR^n$ is a \emph{polytope} if it is the convex hull of a finite subset of points in $\RR^n$.  We say that $P$ is a \emph{$k$-dimensional} polytope, and write $\dim P=k$, if the affine subspace spanned by $P$ has dimension $k$.  A subset $F\subseteq P$ is a \emph{face} of $P$ if there exists a vector $c\in \RR^n$ and a real number $\gamma \in \RR$ such that $c^\top x\leq \gamma$ for every $x \in P$ and
\begin{align}\label{eq:face}
  F=\{x\in P:\: c^\top x= \gamma\},
\end{align}
i.e., there exists an affine subspace separating $F$ from $P$. It is not difficult to check that a face of $P$ is also a polytope. A face of dimension $\ell$ is called an $\ell$-face.  The $0$-faces of a polytope $P$ are called the \emph{vertices} of $P$, while its $1$-faces are the \emph{edges} of $P$.  Throughout the paper, we will often identify a polytope with its set of vertices, since the latter determines the entire polytope.

Given a polytope $P$, we define its graph $G_P$ as the graph whose vertices are the $0$-faces (vertices) of $P$, and whose edges are the $1$-faces (edges) of $P$.  The following result is standard and can be found in \cite{lecturepolytopes}.

\begin{proposition}[\cite{lecturepolytopes}]\label{prop:ziegler}
  The following holds for every polytope $P$:
  \begin{enumerate}[label=(\roman*)]
  \item
    The graph $G_P$ is connected.
  \item
    \label{item:ziegler-ii}
    If $F\subseteq P$ is a face of $P$, then $G_F=G_P[F]$, i.e., the vertices/edges of $F$ are exactly the vertices/edges of $P$ contained in $F$.
  \end{enumerate}
\end{proposition}

We will also use the following geometric observation, whose proof can be found in~\cite{randomexppolytopes24}.

\begin{proposition}[{\cite[Proposition~6]{randomexppolytopes24}}]\label{prop:leroux}
  If $P\subseteq \RR^n$ is a $d$-dimensional polytope, then for any vertex $v \in P$, the set of edges incident to $v$ is not contained in any $(d-1)$-dimensional affine subspace.
\end{proposition}

Recall that $Q^n$ is the $n$-dimensional hypercube with vertex set $\{0,1\}^n$. In this paper, we will often be interested in analyzing the edges of a polytope $P \subseteq Q^n$. It is not hard to check that $G_{Q^n}[P] \subseteq G_P$, i.e., an edge of the hypercube $Q^n$ whose both endpoints are vertices of $P$ is also an edge of $P$. However, unlike in \Cref{prop:ziegler} \ref{item:ziegler-ii}, there may be new edges in $G_P$ that were not previously in $G_{Q^n}$. We conclude this section by giving a sufficient condition for the existence of such new edges.

\begin{proposition}\label{prop:geometry}
  Given integers $1\leq k\leq n$, let $P\subseteq Q^n$ be a polytope, let $F\subseteq Q^n$ be a $k$-face of $Q^n$, and let $x,y \in P\cap F$ be two vertices at Euclidean distance $\sqrt{k}$.  If there is a $(k-1)$-face $F' \subseteq F$ such that $P \cap F' = \{x\}$, then $\{x,y\}\in G_P$.
\end{proposition}

\begin{proof}
  Suppose without loss of generality that $x$ is the zero vector and $y=(y_1, \dotsc, y_n)$, where $y_i=1$ for $1\leq i \leq k$ and $y_i=0$ for $k+1\leq i \leq n$, and that the $(k-1)$-face $F'$ is described by
  \begin{align*}
    F' = \bigl\{(z_1,\ldots,z_n):\: z_i=0 \text{ for } k\leq i \leq n\bigr\}.
  \end{align*}
  Let $c=(c_1,\ldots,c_n)$ be the vector given by $c_i \coloneqq 1$ if $1\leq i \leq k-1$, $c_k \coloneqq -(k-1)$, and $c_i \coloneqq -n$ for $k+1\leq i \leq n$.
  One can easily check that $c^Tx=c^Ty=0$ and that $c^Tz<0$ for all $z \in Q^n \setminus (F' \cup \{y\}) \supseteq P \setminus \{x, y\}$.
  In other words, $\{z : c^Tz=0\}$ is a hyperplane separating $\{x, y\}$ from the rest of the polytope and thus $\{x,y\}$ is an edge of $P$. 
\end{proof}

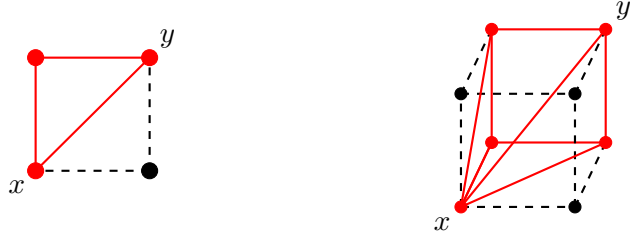
\begin{figure}[h]
  \centering
  {\hfil \begin{tikzpicture}[scale=1.5, x={(1,0,0)}, y={(0,1,0)}, z={(0,-0.3,0.7)}]
      \coordinate (000) at (4,0.25,0);
      \coordinate (100) at (5,0.25,0);
      \coordinate (010) at (4,1.25,0);
      \coordinate (110) at (5,1.25,0);
      \coordinate (001) at (4,0.25,1);
      \coordinate (101) at (5,0.25,1);
      \coordinate (011) at (4,1.25,1);
      \coordinate (111) at (5,1.25,1);

      \draw[thick,red] (0,0) -- (1,1)--(0,1)--cycle;
      \draw[dashed, thick] (1,0) -- (0,0);
      \draw[dashed, thick] (1,0) -- (1,1);

      \filldraw[red] (0,0) circle (2pt);
      \node[below left, black] at (0,0) {$x$};
      \filldraw[red] (0,1) circle (2pt);
      \filldraw[red] (1,1) circle (2pt);
      \node[above right, black] at (1,1) {$y$};
      \filldraw[black] (1,0) circle (2pt);

      \draw[thick, red] (000) -- (100);
      \draw[thick, red] (000) -- (010);
      \draw[thick, red] (000) -- (001);
      \draw[dashed, thick] (100) -- (101);
      \draw[dashed, thick] (010) -- (011);
      \draw[dashed,thick] (001) -- (101);
      \draw[dashed, thick] (001) -- (011);
      \draw[dashed, thick] (101) -- (111);
      \draw[dashed, thick] (011) -- (111);
      \draw[thick, red] (010) -- (110);
      \draw[thick, red] (100) -- (110);
      \draw[dashed, thick] (110) -- (111);
      \draw[thick, red] (001) -- (000);
      \draw[thick, red] (001) -- (010);
      \draw[thick, red] (001) -- (100);
      \draw[thick, red] (001) -- (110);

      \filldraw[red] (000) circle (1.5pt);
      \filldraw[red] (100) circle (1.5pt);
      \filldraw[red] (010) circle (1.5pt);
      \filldraw[red] (110) circle (1.5pt);
      \node[above right, black] at (110) {$y$};
      \filldraw[black] (111) circle (1.5pt) ;
      \filldraw[red] (001) circle (1.5pt);
      \node[below left, black] at (001) {$x$};
      \filldraw[black] (101) circle (1.5pt);
      \filldraw[black] (011) circle (1.5pt);

    \end{tikzpicture}\hfil}
  \caption{An example of \Cref{prop:geometry} for $2$-faces and $3$-faces. The vertices in red are the vertices in $P$. The face $F'$ is the face containing $x$ and all black vertices.}
  \label{fig:fig1}
\end{figure}

\subsection{Projection lemma}\label{subsec:projection}

We devote this subsection to proving our projection lemma. Before stating the lemma, we introduce some notation. Let $d<n$ be integers. Consider the map $\pi \colon Q^n\rightarrow Q^d$ projecting the points of $Q_n$ onto the first $d$ coordinates.
For a $0/1$-polytope $P\subseteq Q^n$, let $R \coloneqq \pi(P)$ be the projection of $P$ on $Q^d$. Furthermore, for every $x \in Q^d$, let $P_x \coloneqq \pi^{-1}(\{x\}) \cap P$ be the fiber of $x$ in $P$.
Given a set $A\subseteq P$, write $A^{\mathsf{c}} \coloneqq P \setminus A$ and define
\begin{equation}\label{eq:BUM}
  \begin{split}
    B\coloneqq \pi(A)=\{x\in Q^d:\: & P_x \cap A \neq \emptyset\}, \\
    U = U_{\pi}(A)\coloneqq\{x\in Q^d:\: & P_x \subseteq A\}, \\
    M = M_{\pi}(A)\coloneqq\{x\in Q^d:\: & P_x \cap A\neq \emptyset \text{ and } P_x \cap A^\mathsf{c}\neq \emptyset\}.
  \end{split}
\end{equation}

In other words, the set $B$ is the projection of $A$ in $R$, the set $U$ is the subset of vertices of $R$ where all the elements projected from $P$ are in $A$, and $M$ is the subset of vertices of $R$ that contain elements projected from both $A$ and $A^{\mathsf{c}}$. Note that $M = B \setminus U$. We are now able to state the lemma.

\begin{lemma}\label{lem:projection}
  Let $A \subseteq P$, and let $R$, $B$, and $M$ be the sets defined above.
  We have
  \[
    e_{G_P}(A,A^{\mathsf{c}})\geq \max\bigl\{|M| ,\, e_{G_R}(B,B^\mathsf{c})\bigr\}.
  \]
\end{lemma}

We remark that \Cref{lem:projection} is similar to the projection lemma in \cite{randomexppolytopes24}. However, because $R$ is not necessarily the hypercube $Q^d$, our statement is slightly more technical. We start with the following geometric observation about projections. 

\begin{claim}\label{clm:face}
  If $F\subseteq R$ is a face of $R$, then the set $\pi^{-1}(F)\cap P$ is a face of $P$.
\end{claim}
\begin{proof}
  Since $F$ is a face of $R$, by definition, there exists a vector $c\in \RR^d$ and $\gamma \in \RR$ such that $c^Tx\leq \gamma$ for every $x\in R$, with equality holding if and only if $x\in F$.
  Let $\tilde{c} \coloneqq (c,\vec{0}) \in \RR^n$, where $\vec{0}$ is the zero vector in $\RR^{n-d}$ and let $z \in P$ be arbitrary.
  We have $\tilde{c}^T z = c^T \pi(z) \le \gamma$ and equality holds if and only if $\pi(z) \in F$, that is, if and only if $z \in \pi^{-1}(F)$.  This implies that $\pi^{-1}(F) \cap P$ is a face of $P$.
\end{proof}

\begin{proof}[Proof of \Cref{lem:projection}]
  We prove the lemma by showing the two estimates separately.
  To see the first inequality, note that \Cref{clm:face} implies that $P_x$ is a face of $P$ for every $x\in R$.
  Further, \Cref{prop:ziegler} asserts that $G_{P_x} = G_P[P_x]$, and thus $G_{P_x} \subseteq G_P$.
  As for every $x\in M$, the face $P_x$ contains vertices from both $A$ and $A^{\mathsf{c}}$, the connectedness of $G_{P_x}$, asserted by \Cref{prop:ziegler}, implies that $G_{P_x}$, and thus also $G_P$, must contain an edge between $A$ and $A^{\mathsf{c}}$.
  Finally, since the graphs $G_{P_x}$, with $x \in R$, are pairwise disjoint, we conclude that $e_{G_P}(A,A^{\mathsf{c}}) \geq |M|$.

  For the second inequality, we will show that, for every edge $\{x,y\} \in G_R$ with $x\in B$ and $y\in B^{\mathsf{c}}$, there is a corresponding edge $\{x', y'\} \in G_P$ with $x' \in P_x \cap A$ and $y' \in P_y \cap A^{\mathsf{c}}$.
  By \Cref{clm:face}, the set $P_{xy} \coloneqq \pi^{-1}(\{x,y\}) \cap P$ is a face of $P$ and $G_{P_{xy}} \subseteq G_P$, by \Cref{prop:ziegler}.
  Moreover, $P_{xy} = P_x \cup P_y$ and $\dim{P_x}, \dim{P_y} < \dim{P_{xy}}$.
  Hence, by \Cref{prop:leroux}, each vertex of $P_y$ has at least one neighbor in $P_x$ in the graph $G_{P_{xy}}$.
  Since $P_x \cap A \neq \emptyset$ and $P_y \subseteq A^{\mathsf{c}}$, we may let $x' \in P_x \cap A$ be arbitrary and $y'$ be one of its neighbors in $P_y$.
  Finally, since the bipartite graphs $G_P[P_x, P_y]$, with $x, y \in R$, are pairwise disjoint, we may conclude that $e_{G_P}(A, A^{\mathsf{c}}) \ge e_{G_R}(B, B^{\mathsf{c}})$.
\end{proof}

\section{The sparse case}\label{sec:sparse}

We now turn our attention to the proof of \Cref{prop:sparse}. First, we remind the reader of the following natural representation of set systems as polytopes:  To each set $A \subseteq [n]$, we associate the characteristic vertex $\1_A \in Q^n$ given by
\[
  \1_A(i) =
  \begin{cases}
    1, &\quad \text{if $i\in A$}\\
    0, &\quad \text{if $i\notin A$}.
  \end{cases}
\]
More generally, given a set system $\Omega \subseteq 2^{[n]}$, we define the characteristic polytope of $\Omega$ by $P_{\Omega}\coloneqq\{\1_A: A\in \Omega\}$. The main ingredient of the proof of \Cref{prop:sparse} is the following observation.

\begin{lemma}\label{clm:complete}
  If $\Omega \subseteq 2^{[n]}$ is a set system with the property that $C \nsubseteq A\cup B$ for all distinct $A,B,C \in \Omega$, then the graph $G_{P_\Omega}$ of the polytope $P_{\Omega}$ is complete. 
\end{lemma}

\begin{proof}
  Suppose that $\Omega$ is a set system satisfying the hypothesis of the lemma.
  Since the statement trivially holds if $|\Omega| \le 1$, we may assume that $\Omega$ has at least two elements.
  Let $A,B \in \Omega$ be any two distinct elements and let $c \coloneqq -\1_{(A \cup B)^{\mathsf{c}}}$.
  The definition of $c$ yields $c^\top \1_A = c^\top\1_B = 0$ while $c^T \1_C < 0$ for every $C \in \Omega \setminus \{A,B\}$, as our hypothesis implies that $C \cap (A \cup B)^{\mathsf{c}} \neq \emptyset$.
  In other words, the hyperplane $\{x : c^Tx = 0\}$ separates $\{\1_A, \1_B\}$ from the remainder of $P_\Omega$.
  This means that $\1_A$ and $\1_B$ form an edge of $P_\Omega$.
  Since $A$ and $B$ were chosen arbitrarily, $G_{P_\Omega}$ is complete.
\end{proof}

The proof of \Cref{prop:sparse} now follows from a first-moment argument.

\begin{proof}[Proof of \Cref{prop:sparse}]
  Let $\Omega \subseteq 2^{[n]}$ be a random set system obtained by selecting every set $A \in 2^{[n]}$ independently with probability $p$; clearly, $P_\Omega \sim P_{n,p}$.
  In view of \Cref{clm:complete}, we just need to prove that, with high probability, the random set system $\Omega$ has the property that $C \nsubseteq A\cup B$ for all distinct $A, B, C \in \Omega$. 

  To this end, define
  \[
    \cT \coloneqq \{(A,B,C) \in (2^{[n]})^3:\: C \subseteq A\cup B\}
  \]
  and note that $|\cT| = 7^n$ as $(A, B, C) \in \cT$ if and only if $\bigl(\1_A(i), \1_B(i), \1_C(i)\bigr) \neq (0,0,1)$ for all $i \in [n]$.
  Hence, the expected number of triples in $\cT\cap \Omega^3$ satisfies
  \[
    \EE\left(|\cT\cap \Omega^3|\right)=p^3|\cT|=p^37^n=o(1),
  \]
  by our assumption on $p$.
  By Markov's inequality, the set system $\Omega$ has no triple $(A,B,C) \in \cT$, that is, $\Omega$ satisfies the hypothesis of \Cref{clm:complete}, with high probability.
\end{proof}

\section{The very dense case}\label{sec:verydense}

The main goal of this section is to prove that very dense random $0/1$-polytopes typically have graphs with constant edge expansion. More precisely, we will show the following statement.

\begin{theorem}\label{thm:verydense}
  Suppose that $t\geq4$ and $R \sim P_{d,q}$ for some $q \geq 1 - e^{-t}$. Then, with high probability, for every $B \subseteq R$ with $|B| \leq 3/4 \cdot 2^d$, we have
  \[
    e_{G_R}(B, B^{\mathsf{c}}) \geq \frac{|B|}{8} \log_2 \left(\frac{2^d}{|B|}\right).
  \]
\end{theorem}

The proof will use the fact that the dense random polytope $R$ is still very close to $Q^d$, and therefore partially inherits its edge expansion properties.
The edge-isoperimetric inequality for the hypercube was proved by Harper~\cite{Harper64} and later reproved by several authors~\cites{B67, Hart76, Lindsey64}.

\begin{theorem}[\cite{Harper64}]
  \label{thm:Harper}
  For all $d \ge 1$ and all $C \subseteq Q^d$, we have
  \[
    e_{Q^d}(C, C^{\mathsf{c}}) \geq |C| \log_2\left(\frac{2^d}{|C|}\right).
  \]
\end{theorem}

Note that \Cref{thm:verydense} can be seen as an edge-isoperimetric inequality for very dense random subpolytopes of $Q^d$, where we relax the condition of being a hypercube at the cost of obtaining a worse constant.

\subsection{Pseudorandom properties of $R$}

We start by defining pseudorandom properties of the polytope $R$ that are sufficient to guarantee good edge expansion.  Throughout the rest of the section, we will write $N \coloneqq 2^d$.  Moreover, for a vertex $x\in Q^d$ and a subset $C\subseteq Q^d$, we will denote by $\deg_{Q^d}(x,C) \coloneqq |N_{Q^d}(x) \cap C|$ the number of neighbors $x$ has in $C$. 

\begin{definition}\label{def:Rproperty}
  Given a $0/1$-polytope $R\subseteq Q^d$ and a real number $\eps>0$, we say that $R$ is \emph{$\eps$-good} if $R$ satisfies the following properties:
  \begin{enumerate}[label=(R\arabic*)]
  \item
    \label{item:epsgood-neighborhood}
    For every vertex $x\in R$, we have $\deg_{Q^d}(x, R) > (1-\eps)d$.
  \item
    \label{item:epsgood-intersection}
    For every set $C\subseteq Q^d$ of size $|C|\geq N/20$, we have $|C\cap R|\geq N/40$.
  \end{enumerate}
\end{definition}

Our first observation is that basic properties of the binomial distribution give that whp a very dense random polytope $R$ is $\eps$-good.

\begin{proposition}\label{prop:Risgood}
  Let $t\geq4$, and let $R\subseteq Q^d$ be a random polytope whose vertices are chosen independently with probability $q\geq 1-e^{-t}$. Then, with high probability, the polytope $R$ is $(2/t)$-good.
\end{proposition}

\begin{proof}
  Since the degree of every vertex $x\in Q^d$ in the random set $R$ follows the binomial distribution $\Bin(d,q)$, we have
  \[
    \PP\bigl(\deg_{Q^d}(x, R) \le (1-2/t)d\bigr) \leq \binom{d}{2d/t} (1-q)^{2d/t} \leq \left(\frac{et}{2}\right)^{2d/t} e^{-2d} \le e^{-d},  \]
  where we used our assumptions that $1-q\leq e^{-t}$ and $t\geq4$.  By the union bound, $R$ violates \ref{item:epsgood-neighborhood} with probability at most $Ne^{-d}=o(1)$.  Further, since whp
  \[
    |R^{\mathsf{c}}| \le (1-q+o(1)) \cdot N \le (e^{-t} + o(1)) \cdot N \le N/40,
  \]
  for every $C \subseteq Q^d$ with $|C| \ge N/20$, we have $|C \cap R| \ge |C| - |R^{\mathsf{c}}| \ge N/20$, proving~\ref{item:epsgood-intersection}.
\end{proof}

\subsection{Deterministic Lemma}

In this subsection, we present the main technical result of the paper.  Given an $\eps$-good polytope $R \subseteq Q^d$ and a set $B \subseteq R$, our goal is to find a good lower bound on $e_{G_R}(B,B^\mathsf{c})$.  Before stating the lemma precisely, we define the following subsets of vertices in $Q^d$:
\begin{align}\label{eq:SLX}
  S\coloneqq S(B,\eps)&=\{x\in Q^d\setminus B:\: \deg_{{Q^d}}(x, B)\leq (1-2\eps)d\},\nonumber\\
  L\coloneqq L(B,\eps)&=\{x\in Q^d \setminus B:\: \deg_{{Q^d}}(x, B)> (1-2\eps)d\},\\
  X\coloneqq X(B,\eps)&=\{x\in Q^d \setminus (B\cup L):\: \deg_{Q^d}(x,L)>0\}\nonumber.
\end{align}
That is, $S$ is the set of vertices with small degree into $B$, $L$ is the set of vertices with large degree into $B$, and $X$ is the set of vertices not in $B$ that have a neighbor in $L$ (note that $S$ and $X$ are not necessarily disjoint). Moreover, note that the values of $S$, $L$, and $X$ depend only on the set $B$ and the hypercube $Q^d$, but not on the polytope $R$. Our main deterministic lemma is stated as follows.

\begin{lemma}\label{lem:deterministic}
  Let $\eps \in (0, 1/4)$, and suppose that $R\subseteq Q^d$ is $\eps$-good.  Then, for any $B\subseteq R$,
  \[
    e_{G_R}(B,B^\mathsf{c})\geq \max\left\{\frac{e_{Q^d}(B,S)}{2},\, \frac{e_{Q^d}(L,X)}{4}\right\},
  \]
  provided that $d$ is sufficiently large (as a function of $\eps$ only).
\end{lemma}

\begin{proof}

  We start by noticing that every vertex $x \notin B\cup L$ has many $Q_d$-neighbors in $R \setminus B$. 

  \begin{claim}\label{clm:posdegree}
    If $x \notin B \cup L$, then $\deg_{Q^d}(x, R \setminus B) \geq \eps d$.
  \end{claim}
  \begin{proof}
    Since $R$ is $\eps$-good and $x \notin B \cup L$, we have
    \[
      \deg_{Q^d}(x, R \setminus B) \geq \deg_{Q^d}(x,R) - \deg_{Q^d}(x,B) \geq (1-\eps)d - (1-2\eps)d = \eps d,
    \]
    as desired.
  \end{proof}

  We now prove each inequality in the statement separately.

  \begin{claim}
    $e_{G_R}(B,B^{\mathsf{c}}) \geq e_{Q^d}(B,S)/2$.
  \end{claim}

  \begin{proof}
    In order to prove the asserted inequality, it suffices to construct a map $\Psi_1 \colon E_{Q^d}(B,S)\rightarrow E_{G_R}(B,B^{\mathsf{c}})$ satisfying $|\Psi_1^{-1}(e)| \le 2$ for all $e$.
    First, for each $s \in S$, let $\phi(s) \in R \cap B^{\mathsf{c}}$ be an arbitrary vertex such that $\{s,\phi(s)\} \in E(Q^d)$; the existence of such a vertex is guaranteed by \Cref{clm:posdegree}.
    Now, fix some $b \in B$ and $s \in S$ that are adjacent in $Q^d$ (\Cref{fig:fig2}).
    Since $\phi(s) \notin B$, we have $\phi(s)\neq b$ and, consequently, $\dist_{Q^d}(b,\phi(s))=2$.
    Let $F_2$ be the unique $2$-face of $Q^d$ containing the vertices $\{b,s,\phi(s)\}$. There are two possibilities:

    \begin{figure}[h]
      \centering
      {\hfil \begin{tikzpicture}[scale=1.5, x={(1,0,0)}, y={(0,1,0)}, z={(0,-0.3,0.7)}]
          \draw[dashed, thick] (0,0) --(0,1)--(1,1);
          \draw[dashed, thick] (1,0) -- (0,0);
          \draw[dashed, thick] (1,0) -- (1,1);

          \filldraw[black] (0,0) circle (2pt);
          \node[below left, black] at (0,0) {$b \in R\cap B$};
          \filldraw[black] (0,1) circle (2pt);
          \node[above left, black] at (0,1) {$s\in B^{\mathsf{c}}$};
          \filldraw[black] (1,1) circle (2pt);
          \node[above right, black] at (1,1) {$\phi(s) \in R\cap B^{\mathsf{c}}$};
          \filldraw[black] (1,0) circle (2pt);
        \end{tikzpicture}\hfil}
      \caption{The $2$-face $F_2$ containing $\{b,s,\phi(s)\}$.}
      \label{fig:fig2}
    \end{figure}
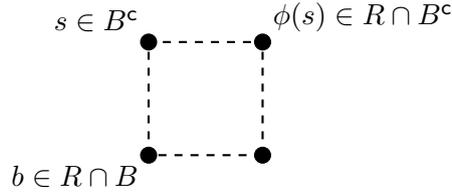

    \begin{enumerate}[label=(\roman*)]
    \item
      \underline{$s\in R$}:
      The edge $\{b,s\}$ is an edge of $G_R$ from $B$ to $B^\mathsf{c}$.  We set $\Psi_1(\{b,s\}) \coloneqq \{b,s\}$.
    \item
      \underline{$s\notin R$}:
      \Cref{prop:geometry} with $k=2$, $F=F_2$, and $F'=\{b,s\}$ yields that $\{b,\phi(s)\} \in E(G_R)$. We set $\Psi_1(\{b,s\}) \coloneqq \{b,\phi(s)\}$.
    \end{enumerate}

    Let $\{u,v\}$ be an edge in the image of $\Psi_1$.
    To show that $|\Psi_1^{-1}(\{u,v\})|\leq 2$, we consider two cases.
    If $\dist_{Q^d}(u,v)=1$, then $\Psi_1^{-1}(\{u,v\}) = \{\{u,v\}\}$.
    Otherwise, if $\dist_{Q^d}(u,v)=2$, we have $\{u,v\} = \{b,\phi(s)\}$ for some $s\in S$ that is a common neighbor of both $u$ and $v$;  thus $|\Psi_1^{-1}(u,v)|\leq 2$.
    This concludes the proof.
  \end{proof}

  The proof of the second inequality is similar but slightly more technical.

  \begin{claim}
    $e_{G_R}(B,B^{\mathsf{c}}) \ge e_{Q^d}(L,X)/4$.
  \end{claim}

  \begin{proof}
    Let $\cC$ be the family of ordered $2$-faces $(y_1,y_2,y_3,y_4)$ of $Q^d$ that satisfy:
    \begin{enumerate}[label=(\alph*)]
    \item 
      $\{y_1,y_2\}$, $\{y_2,y_3\}$, $\{y_3,y_4\}$, and $\{y_1,y_4\}$ are edges of $Q^d$ and
    \item 
      $y_1\in B$, $y_2\in L$, and $y_3 \in X$
    \end{enumerate}
    and observe that
    \begin{equation}
      \label{eq:cC-lower}
      |\cC| = \sum_{\substack{y_2y_3 \in E(Q^d) \\ y_2 \in L, y_3 \in X}} \deg_{Q^d}(y_2, B) \ge e_{Q^d}(L, X) \cdot (1-2\eps)d \ge \frac{d}{2} \cdot e_{Q^d}(L, X),
    \end{equation}
    where we used that $\deg_{Q^d}(y_2,B)\geq (1-2\eps)d$ for $y_2\in L$, see \eqref{eq:SLX}.
    In order to prove the asserted inequality, it thus suffices to construct a map $\Psi_2 \colon \cC \rightarrow E_{G_R}(B,B^{\mathsf{c}})$ with $|\Psi_2^{-1}(e)| \le 2d$ for all $e$.

    To this end, for each face $C = (y_1,y_2,y_3,y_4) \in \cC$, let $\phi(C) \in B^{\mathsf{c}}\cap R$ be a vertex such that $\{y_3,\phi(C)\} \in E(Q^d)$ and $\phi(C)\notin \{y_2,y_4\}$ (\Cref{fig:fig3};
    the existence of such a vertex is guaranteed by \Cref{clm:posdegree}, as $X \subseteq (B \cup L)^{\mathsf{c}}$, provided that $d$ is sufficiently large as a function of $\eps$.
    Note that $\dist_{Q^d}(y_1, \phi(C))=3$, and let $\tilde{C}$ be the unique $3$-face of $Q^d$ that contains both $C$ and $\phi(C)$.

    \begin{figure}[h]
      \centering
      {\hfil \begin{tikzpicture}[scale=1.5, x={(1,0,0)}, y={(0,1,0)}, z={(0,-0.3,0.7)}]
          \coordinate (000) at (4,0.25,0);
          \coordinate (100) at (5,0.25,0);
          \coordinate (010) at (4,1.25,0);
          \coordinate (110) at (5,1.25,0);
          \coordinate (001) at (4,0.25,1);
          \coordinate (101) at (5,0.25,1);
          \coordinate (011) at (4,1.25,1);
          \coordinate (111) at (5,1.25,1);

          \draw[dashed, thick] (000) -- (100);
          \draw[dashed, thick] (000) -- (010);
          \draw[dashed, thick] (000) -- (001);
          \draw[dashed, thick] (100) -- (101);
          \draw[dashed, thick] (010) -- (011);
          \draw[dashed, thick] (001) -- (101);
          \draw[dashed, thick] (001) -- (011);
          \draw[dashed, thick] (101) -- (111);
          \draw[dashed, thick] (011) -- (111);
          \draw[dashed, thick] (010) -- (110);
          \draw[dashed, thick] (100) -- (110);
          \draw[dashed, thick] (110) -- (111);

          \filldraw[black] (000) circle (1.5pt);
          \filldraw[black] (100) circle (1.5pt);
          \filldraw[black] (010) circle (1.5pt);
          \filldraw[black] (110) circle (1.5pt);
          \node[above right, black] at (110) {$\phi(C)$};
          \filldraw[black] (111) circle (1.5pt) ;
          \node[above left, black] at (111) {$y_3$};
          \filldraw[black] (001) circle (1.5pt);
          \node[below left, black] at (001) {$y_1$};
          \filldraw[black] (101) circle (1.5pt);
          \node[below right, black] at (101) {$y_4$};
          \filldraw[black] (011) circle (1.5pt);
          \node[above left, black] at (011) {$y_2$};

        \end{tikzpicture}\hfil}
      \caption{The $3$-face $\tilde{C}$ containing $C=\{y_1,y_2,y_3,y_4\}$ and $\phi(C)$.}
      \label{fig:fig3}
    \end{figure}
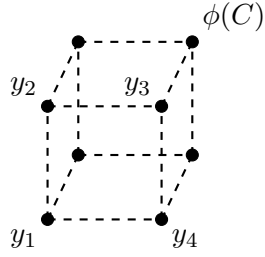

    There are several possibilities:
    \begin{enumerate}[label=(\roman*)]
    \item
      \underline{$y_2\in R$}: The edge $\{y_1, y_2\}$ belongs to $E_{G_R}(B, B^{\mathsf{c}})$.  We set $\Psi_2(C) \coloneqq \{y_1,y_2\}$.
    \item 
      \underline{$y_2\notin R$, $y_3\in R$}: \Cref{prop:geometry} with $k=2$, $F=C$, and $F'=\{y_1,y_2\}$ yields that $\{y_1, y_3\} \in E(G_R)$. We set $\Psi_2(C) \coloneqq \{y_1,y_3\}$.
    \item
      \underline{$y_2,y_3\notin R$, $y_4\in R \cap B^{\mathsf{c}}$}: The edge $\{y_1, y_4\}$ belongs to $E_{G_R}(B, B^{\mathsf{c}})$.  We set $\Psi_2(C) \coloneqq \{y_1,y_4\}$.
    \item 
      \underline{$y_2,y_3\notin R$, $y_4\in R\cap B$}: Let $F$ be the unique $2$-face containing both $y_4$ and $\phi(C)$.  \Cref{prop:geometry} with $k=2$, $F$ defined above, and $F' = \{y_3,y_4\}$ yields that  $\{y_4, \phi(C)\} \in E(G_R)$. We set $\Psi_2(C) \coloneqq \{y_4,\phi(C)\}$.
    \item 
      \underline{$y_2,y_3,y_4 \notin R$}: \Cref{prop:geometry} with $k=3$, $F=\tilde{C}$, and $F'=C$ yields that $\{y_1,\phi(C)\} \in E(G_R)$.  We set $\Psi_2(C) \coloneqq \{y_1,\phi(C)\}$.
    \end{enumerate}

    Let $\{u,v\}$ be an edge in the image of $\Psi_2$.
    To show that $|\Psi_2^{-1}(\{u,v\})|\leq 2d$, as in the previous proof, we consider several cases, depending on the distance between $u$ and $v$.

    If $\dist_{Q^d}(u,v)=1$, then either $\{u,v\} = \{y_1,y_2\}$ or $\{u,v\}=\{y_1,y_4\}$. In each of the two subcases, there are at most $d-1$ choices for $y_3$ (which then determines the remaining of $C$).  Thus, $|\Psi_2^{-1}(\{u,v\})| \le 2(d-1)$.

    If $\dist_{Q^d}(u,v)=2$, then either $\{u,v\}=\{y_1,y_3\}$ or $\{u,v\}=\{y_4,\phi(C)\}$.  In the first subcase, there are at most two ways to choose the ordered face $C$ (since there is only one $2$-face containing $\{y_1,y_3\}$). In the latter subcase, there are two choices for the vertex $y_3$ (the common neighbors of $y_4$ and $\phi(C)$) and, for each of those, at most $d-2$ choices for $y_2$, which then uniquely determines $C$.  Thus, $|\Psi_2^{-1}(\{u,v\})| \leq 2 + 2(d-2) \le 2d$.

    Finally, if $\dist_{Q^d}(u,v)=3$, we have $\{u,v\}=\{y_1,\phi(C)\}$.  There is a unique $3$-face $\tilde{C}$ that contains both $u$ and $v$ and two choices to decide which of those two vertices is $y_1$.  Finally, there are $6$ further choices for an ordered face $C$ that contains $y_1$. This implies that $|\Psi_2^{-1}(\{u,v\})|\leq 12 \le 2d$.
  \end{proof}

  The proof of the lemma now follows by combining the last two claims.
\end{proof}

\subsection{Proof of \Cref{thm:verydense}}

We are now able to prove the main theorem of this section. Note that the next theorem, combined with \Cref{prop:Risgood}, implies \Cref{thm:verydense}.

\begin{theorem}\label{thm:R}
  Let $\eps \in (0, 1/4)$, let $d \ge 10$, and suppose that $R \subseteq Q^d$ is $\eps$-good. Then, for every $B\subseteq R$ with $|B| \le 3N/4$,
  \[
    e_{G_R}(B,B^{\mathsf{c}})\geq \frac{|B|}{8} \log_2\left(\frac{N}{|B|}\right).
  \]
\end{theorem}

\begin{proof}
  Fix some $B\subseteq R$ with $|B| \le 3N/4$ and let $S$, $L$, and $X$ be the sets defined as in \eqref{eq:SLX}.
  We consider two cases, depending on the size of $L$.

  \smallskip
  \noindent\underline{Case 1}: $|L|\geq N/20$.
  \smallskip

  Since $R$ is $\eps$-good, by property \ref{item:epsgood-intersection}, we have that $|L\cap R|\geq N/40$. Since $G_{Q^d}[C \cap R] \subseteq G_R[C \cap R]$ for every $C \subseteq Q^d$, we obtain that
  \[
    e_{G_R}(B,B^{\mathsf{c}}) \geq e_{Q^d}(B, L \cap R)\geq |L \cap R|\cdot (1-2\eps)d \geq \frac{dN}{80} \ge \frac{N}{8} \ge \frac{|B|}{8}\log_2\left(\frac{N}{|B|}\right),
  \]
  where we used that $\deg_{Q^d}(y, B) \ge (1-2\eps)d$ for every $y \in L$ as well as the assumption on $d$.

  \smallskip
  \noindent\underline{Case 2}: $|L|< N/20$.
  \smallskip

  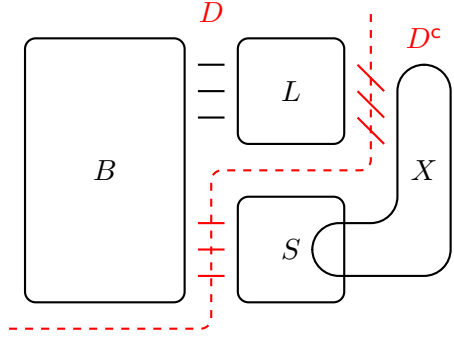
\begin{figure}[h]
    \centering
    {\hfil\begin{tikzpicture}[scale=0.7]


        \draw[rounded corners, thick] (-1,0) rectangle (2,5) node[midway] {$B$};
        \draw[rounded corners, thick] (3,0) rectangle (5,2) node[midway] {$S$};
        \draw[rounded corners, thick] (3,3) rectangle (5,5) node[midway] {$L$};

        \draw[thick, rounded corners=10pt] (4.4, 0.5) -- (4.4, 1.5) -- (6, 1.5) -- (6, 4.5) -- (7, 4.5) -- (7, 0.5) -- cycle;
        \node at (6.5, 2.5) {$X$};

        \draw[thick,rounded corners, dashed, red] (-1.3, -0.5) -- (2.5, -0.5) -- (2.5, 2.5) -- (5.5, 2.5) -- (5.5, 5.5);
        \node[red] at (2.5, 5.5) {$D$};
        \node[red] at (6.5, 5) {$D^{\mathsf{c}}$};

        \draw[thick, red] (2.25,0.5) -- (2.75,0.5);
        \draw[thick, red] (2.25,1) -- (2.75,1);
        \draw[thick, red] (2.25,1.5) -- (2.75,1.5);

        \draw[thick] (2.25,3.5) -- (2.75,3.5);
        \draw[thick] (2.25,4) -- (2.75,4);
        \draw[thick] (2.25,4.5) -- (2.75,4.5);

        \draw[thick, red] (5.25,3.5) -- (5.75,3);
        \draw[thick, red] (5.25,4) -- (5.75,3.5);
        \draw[thick, red] (5.25,4.5) -- (5.75,4);
      \end{tikzpicture}\hfil}
    \caption{The sets $B$, $D$, $L$, $S$ and $X$ and the edges from $e_{G_Q^d}(D, D^{\mathsf{c}})$ in red.}
    \label{fig:fig4}
  \end{figure}
  Let $D \coloneqq B \cup L$, and note that $e_{Q^d}(D,D^{\mathsf{c}})=e_{Q^d}(B,S)+e_{Q^d}(L,X)$ (\Cref{fig:fig4}).  \Cref{lem:deterministic} and the edge-isoperimetric inequality for $Q^d$ (\Cref{thm:Harper}) imply that
  \begin{equation}\label{eq:verydense}
    e_{G_R}(B,B^{\mathsf{c}})\geq \max\left\{\frac{e_{Q^d}(B,S)}{2},\, \frac{e_{Q^d}(L,X)}{4}\right\} \ge \frac{e_{Q^d}(D, D^{\mathsf{c}})}{6} \geq \frac{|D|}{6}\log_2\left(\frac{N}{|D|}\right).
  \end{equation}
  Finally, define the function $f \colon [0,N] \to [0,\infty)$ by $f(x) \coloneqq x \log_2(N/x)$ and note that $f$ is concave and positive. Therefore, the function $f(x+a)/f(x)$ is decreasing for $a>0$. Since $|B| \le 3N/4$ and $|B| \le |D| \le |B| + |L| \le |B| + N/20$, we have
  \[
    f(|D|) = \frac{f(|D|)}{f(|B|)} \cdot f(|B|) \ge \frac{f(3N/4+|D|-|B|)}{f(3N/4)} \cdot f(|B|) \ge \frac{4f(|B|)}{5}.
  \]
  Using~\eqref{eq:verydense}, we may finally conclude that
  \[
    e_{G_R}(B, B^{\mathsf{c}}) \ge \frac{f(|D|)}{6} \ge \frac{f(|B|)}{8} = \frac{|B|}{8}\log_2\left(\frac{N}{|B|}\right),
  \]
  as desired.
\end{proof}

\newpage

\section{The general case}\label{sec:general}

In this section, we prove \Cref{thm:main}. As previously discussed in the proof overview, our approach is to project the random polytope $P$ onto its first $d$ coordinates, where $d$ is chosen such that $p2^{n-d}$ is a sufficiently large constant. The projected polytope $R$ will have good edge expansion by \Cref{thm:verydense}. This, together with the projection lemma (\Cref{lem:projection}), will yield our desired bounds. We now proceed to the details.

\subsection{Random Properties of the Projection}\label{subsec:randprojection}

Let $P \sim P_{n,p}$ and let $d$ be the unique integer such that
\begin{equation}
  \label{eq:t}
  \max\{8, t_0\} \le t \coloneqq p2^{n-d} < 2\max\{8, t_0\},
\end{equation}
where $t_0$ is the constant from the assertion of \Cref{lemma:binomial-tilt}.
Note that our assumption that $p \geq 2^{-0.99n}$ implies that
\begin{equation}
  \label{eq:d}
  n/200 \leq d < n
\end{equation}
for sufficiently large $n$.
As usual, let $\pi \colon Q^n \to Q^d$ denote the projection onto the first $d$ coordinates, let
\[ N:=2^d,\] and let $R \coloneqq \pi(P)$ be the projected polytope.  For each nonnegative integer $\ell$, let $X_\ell$ denote the set of all $x \in Q^d$ whose fiber has size $\ell$ and let $\xi_\ell$ be the probability that a given vertex of $Q^d$ belongs to $X_\ell$, i.e.,
\begin{equation}
  \label{eq:Xl}
  X_\ell \coloneqq \{x \in Q^d : |P_x| = \ell\}
  \qquad
  \text{and}
  \qquad
  \xi_\ell \coloneqq \PP\bigr(\Bin(2^{n-d},p) = \ell\bigr).
\end{equation}
Further, for every $L \ge 0$, denote
\[
  X_{\le L} \coloneqq \bigcup_{\ell \le L} X_\ell
  \qquad
  \text{and}
  \qquad
  \xi_{\le L} \coloneqq \sum_{\ell \le L} \xi_\ell.
\]
We now describe the random properties necessary for the proof.

\begin{proposition}\label{prop:Pprop}
  With high probability, $P$ has the following properties:
  \begin{enumerate}[label=(P\arabic*)]
  \item
    \label{item:P-size}
    $|P| = (1 + o(1)) \cdot p2^n = (1+o(1)) \cdot t 2^d$;
  \item
    \label{item:P-fibers-empirical-dist}
    For every $L \ge 0$, $\bigl||X_{\le L}| - \xi_{\le L} \cdot N\bigr| \le N^{2/3}$.
  \item
    \label{item:P-fibers-tail}
    $|X_\ell| \le N\cdot(2et/\ell)^\ell$ for every $\ell \ge 2et$;
  \item
    \label{item:P-pi-good}
    $R = \pi(P)$ is $(2/t)$-good;
  \end{enumerate}
\end{proposition}

\begin{proof}
  To see~\ref{item:P-size}, note that, by our assumption on $p$, we have $\mathbb{E}[|P|] = p2^n \geq 2^{0.01n}$ and that $\Var(|P|) = p(1-p)2^n \le p2^n$.
  Therefore, by Chebyshev's inequality, for every $\delta > 0$,
  \[
    \mathbb{P}\left(\bigl||P| - p2^n\bigr| > \delta p2^n\right) \leq \frac{1}{\delta^2p2^n}.
  \]

  To see that the remaining three properties hold whp, observe first that the random variables $\{|P_x|\}_{x \in Q^d}$ are independent and follow the binomial distribution $\Bin(2^{n-d}, p)$.
  In view of this, property~\ref{item:P-fibers-empirical-dist} is an immediate consequence of \Cref{thm:DKW}.

  To see \ref{item:P-fibers-tail}, observe that $\xi_\ell \le (et/\ell)^\ell$ for all $\ell \ge 2et$, by~\Cref{chernoff}.
  If $\ell \le \frac{d}{5\log d}$, say, it follows from \ref{item:P-fibers-empirical-dist} that with high probability
  \[
    |X_\ell| \le N \cdot \xi_\ell + 2N^{2/3} \le N \cdot \left(\left(\frac{et}{\ell}\right)^\ell + 2^{1-d/3}\right) \le N \cdot \left(\frac{2et}{\ell}\right)^\ell.
  \]
  Otherwise, if $\ell > \frac{d}{5\log d}$, then we may simply use Markov's inequality to deduce that
  \[
    \PP\left(|X_\ell| \ge N \cdot \left(\frac{2et}{\ell}\right)^\ell\right) \le \frac{\mathbb{E}[|X_\ell|]}{N \cdot (2et/\ell)^\ell} = \frac{N \cdot \xi_\ell}{N \cdot (2et/\ell)^\ell} \le 2^{-\ell}.
  \]
  and apply the union bound over all $\ell > \frac{d}{5\log d}$ to conclude that $|X_\ell| \le N \cdot (2et/\ell)^\ell$ with high probability.

  Finally, to see \ref{item:P-pi-good}, note that $x \in R$ if and only if $|P_x| > 0$ and thus $R$ is a $(1-\xi_0)$-random subset of $Q^d$.  Since $\xi_0 = (1-p)^{2^{n-d}} \le e^{-p2^{n-d}} \le e^{-t}$, we may invoke \Cref{prop:Risgood} to conclude that $R$ is $(2/t)$-good with high probability. 
\end{proof}

\subsection{Proof of \Cref{thm:main}}

In this subsection, we present the proof of \Cref{thm:main}. Let $P$, $R$, $d$, $t$, and $\pi$ be defined as in \Cref{subsec:randprojection} and assume that $P$ has properties \ref{item:P-size}--\ref{item:P-pi-good} from \Cref{prop:Pprop}.
Fix some set $A \subseteq P$ and and let $B$, $U$, and $M$ be defined as in \eqref{eq:BUM}, i.e.,
\begin{align*}
  \begin{split}
    B\coloneqq \pi(A)=\{x\in Q^d:\: & P_x \cap A \neq \emptyset\}, \\
    U = U_{\pi}(A)\coloneqq\{x\in Q^d:\: & P_x \subseteq A\}, \\
    M = M_{\pi}(A)\coloneqq\{x\in Q^d:\: & P_x \cap A\neq \emptyset \text{ and } P_x \cap A^\mathsf{c}\neq \emptyset\}.
  \end{split}
\end{align*}

If $t$ were a sufficiently large function of $n$, concentration inequalities would guarantee that whp all fibers have sizes close to $t$.  In particular, this would imply that $|A| \le (1+o(1))t|B|$ and that $|U| \le (1+o(1))|A|/t \le (1/2+o(1))N$.  However, our $t$ is a constant independent of $n$, so we cannot draw such conclusions easily.  Nevertheless, one can still obtain a slightly weaker upper bound on $|U|$ and a relatively good lower bound on the ratio~$|B| / |A|$.

\begin{claim}\label{clm:U}
  If $|A| \leq |P|/2$, then $|U| \leq 3N/5$.
\end{claim}

\begin{proof}
  Suppose that $|U| > 3N/5$ and observe that
  \[
    |A| \ge \sum_{x \in U} |P_x| = \sum_{\ell \ge 1} \ell \cdot |X_\ell \cap U| \ge \sum_{\ell = 1}^m \ell \cdot |X_\ell|,
  \]
  where $m$ is the largest integer such that $|X_{\le m}| \le 3N/5$.
  Let $T \sim \Bin(2^{n-d}, p)$.
  Property \ref{item:P-fibers-empirical-dist} and the maximality of $m$ imply that
  \[
    \PP(T \le m) = \xi_{\le m} \le 3/5 + o(1)
    \qquad
    \text{and}
    \qquad
    \PP(T \le m+1) = \xi_{\le m+1} \ge 3/5 - o(1),
  \]
  so in particular $m \le 3t$, as $\PP(T > 3t) \le 1/3$.
  Finally, by \Cref{lemma:binomial-tilt} and \eqref{eq:t},
  \[
    \sum_{\ell=1}^m \ell \cdot |X_\ell| \ge \sum_{\ell=1}^m \ell \cdot \xi_\ell \cdot N - o(mN) = \mathbb{E}[T \cdot \1_{T \le m}] \cdot N - o(tN) \ge 5Nt/9 - o(tN),
  \]
  which implies that $|A| > |P|/2$.
\end{proof}

\begin{claim}\label{clm:B}
  For every $A \subseteq P$ and $B = \pi(A)$ such that $|B| \leq 3N / 4$, we have
  \begin{align*}
    |A| \leq |B| \cdot \frac{C_t \log ( N/|B|)}{\log \left(C_t\log\left(N/|B|\right)\right)},
  \end{align*}
  where $C_t$ is a constant depending only on $t$.
\end{claim}

\begin{proof}
  Let $L \coloneqq \lceil 4et \rceil$ and let $f \colon [L, \infty) \to \mathbb{R}$ be the function defined by $f(x) \coloneqq (x/(4et))^x$.  One can check that $f$ is strictly increasing and convex and therefore $f^{-1} \colon [f(L), \infty) \to [L, \infty)$ is well-defined and concave.
  Define, for each positive integer $\ell$,
  \[
    b_\ell \coloneqq |\{x \in B : |P_x| = \ell\}|,
  \]
  and observe that, writing $\ell \vee L$ for $\max\{\ell, L\}$,
  \[
    \begin{split}
      |A|
      & = \sum_{x \in B} |\pi^{-1}(\{x\}) \cap A| \le \sum_{x \in B} |P_x| = \sum_{\ell \ge 1}b_\ell\cdot \ell \le \sum_{\ell \ge 1} b_\ell \cdot (\ell \vee L) \\
      & = \sum_{\ell \ge 1} b_\ell \cdot f^{-1}(f(\ell \vee L)) \le |B| \cdot f^{-1}\left( \sum_{\ell \ge 1} \frac{b_\ell}{|B|} \cdot f(\ell \vee L)\right),
    \end{split}
  \]
  where the last inequality follows from concavity of $f^{-1}$ and the fact that $b_1 + b_2 + \dotsb = |B|$.
  Since $b_{\ell}\leq |X_{\ell}|\leq 2^{-\ell} \cdot N / f(\ell)$ for all $\ell\geq L$, by item~\ref{item:P-fibers-tail} in \Cref{prop:Pprop}, we have
  \[
    \sum_{\ell \ge 1} b_\ell \cdot f(\ell \vee L) \le |B| \cdot f(L) + N \cdot \sum_{\ell > L} 2^{-\ell} = |B| \cdot f(L) + 2^{-L} \cdot N.
  \]
  We may thus conclude that
  \[
    |A| \le |B| \cdot f^{-1}\left(f(L) + 2^{-L} \cdot \frac{N}{|B|}\right).
  \]
  It thus suffices to argue that there is a constant $C$ such that, for all $x \ge 4/3$,
  \begin{equation}
    \label{eq:B-f-sufficient}
    f^{-1} \left( f(L) + 2^{-L} \cdot x \right) \le \frac{ C\log x}{\log (C\log x)}.
  \end{equation}
  Indeed, when $C$ is large, we have $4et \log (C \log x) \le (C \log x)^{1/2}$ for all $x \ge 4/3$, and thus
  \[
    f\left(\frac{ C\log x}{\log (C\log x)}\right) \ge \left(C \log x\right)^{\frac{1}{2} \cdot \frac{C \log x}{\log (C \log x)}} = x^{C/2} \ge f(L) + 2^{-L} \cdot x,
  \]
  which implies~\eqref{eq:B-f-sufficient}, as $f$ is strictly increasing.
\end{proof}

We are now able to complete the proof of \Cref{thm:main}.
Let $C_t$ be the constant from \Cref{clm:B}.
We consider two cases:

\vspace{0.2cm}

\noindent \underline{Case 1}: $|M| \ge N / 20$

\vspace{0.2cm}

By \Cref{lem:projection} and \Cref{prop:Pprop} \ref{item:P-size},
\[
  e_{G_P}(A, A^{\mathsf{c}}) \ge |M| \ge \frac{N}{20} = \frac{(1+o(1))|P|}{20t} \ge \frac{|A|}{11t}.
\]

\vspace{0.2cm}

\noindent \underline{Case 2}: $|M| < N / 20$

\vspace{0.2cm}

Since $B$ is the union of $U$ and $M$ and $|U| \le 3N/5$, by \Cref{clm:U}, we have $|B| \le 3N/5 + N/20 \leq 3N/4$.
We may thus use \Cref{lem:projection} and \Cref{thm:R} to conclude that
\begin{align}\label{eq:concluding}
  e_{G_P}(A, A^{\mathsf{c}}) \geq e_{G_R}(B, B^{\mathsf{c}}) \geq \frac{|B|}{8}\log_2\left(\frac{N}{|B|}\right) \ge \frac{|A|}{8C_t}\log\left(C_t\log\left(\frac{N}{|B|}\right)\right)\geq \frac{|A|}{8C_t},
\end{align}
where the second to last inequality follows from \Cref{clm:B}. This concludes the first part of \Cref{thm:main}. 

To check the second part of the theorem, choose $\eta>0$ sufficiently small so that
\[
  \eta\leq \frac{1}{3t} \qquad \text{and} \qquad \log\left(C_t\log\left(\frac{1}{2\eta t}\right)\right) \geq 8C_t,
\]
and suppose that $A\subseteq P$ satisfies $|A|\leq \eta |P|$.  By \Cref{prop:Pprop} \ref{item:P-size}, it holds that
\[
  |B| \leq |A| \le \eta \cdot 2tN\leq 3N/4.
\]
We may thus use \Cref{lem:projection} and \Cref{thm:R} to conclude, as in \eqref{eq:concluding}, that
\begin{align}\label{eq:concluding2}
  e_{G_P}(A, A^{\mathsf{c}}) \geq \frac{|A|}{8C_t}\log\left(C_t\log\left(\frac{N}{|B|}\right)\right) \ge \frac{|A|}{8C_t} \log\left(C_t\log\left(\frac{1}{2\eta t}\right)\right) \ge |A|.
\end{align}

\section{Remarks}\label{sec:remarks}

In this paper, we studied the edge expansion of graphs arising from random $0/1$-polytopes. We showed that, with high probability, the edge expansion is bounded below by a positive constant. This result aligns with a broader perspective related to the Mihail--Vazirani conjecture, which asserts that every $0/1$-polytope has edge expansion at least $1$ --- a bound that is tight, as witnessed by the $n$-dimensional Boolean hypercube.

Our findings suggest that the edge expansion of random $0/1$-polytopes may be significantly better.  Indeed, our analysis in the last part of the proof (see \eqref{eq:concluding2}) shows the following stronger statement holds with high probability in the graph of $P \sim P_{n,p}$, provided that $p = p(n) \ge 2^{-0.99n}$: For every $K>0$, there exists $\eta = \eta(K)>0$ such that every $A \subseteq P$ satisfying $|A|\leq \eta|P|$ has edge expansion at least $K$. We believe that, as long as the parameter $p$ is bounded away from~$1$, the edge expansion of the graph of $P_{n,p}$ is not just bounded from below by $1$, but actually becomes larger, and possibly tends to infinity as $n \to \infty$.
As a concrete first step in this direction, we propose the following question:

\begin{question}
  Is it true that for \( p = o(1) \) and $p 2^n\rightarrow \infty$, with high probability, \( h(G_{P_{n,p}}) \to \infty \) as \( n \to \infty \)?
\end{question}

We remark that \Cref{prop:sparse} answers this question affirmatively if $p \le c^{-n}$ for some $c < 7^{-1/3}$. For larger values of $p$, a possible heuristic goes as follows. Suppose that $p \le n^{-C}$ for some sufficiently large constant $C>0$ and choose $d$ so that $p \le c^{n-d}$ for some $c<7^{-1/3}$. Then an application of \Cref{prop:sparse} shows that, with high probability, the graphs of most fibers $P_x$ are complete for $x\in Q^d$. Moreover, one can check that $P_{xy}$ is typically complete for most pairs $\{x,y\}\in G_{Q^d}$. These observations suggest that, in this regime, the graph of the polytope becomes significantly more well structured, and perhaps this could lead to an affirmative answer to the question. We also make the following conjecture:

\begin{conjecture}
  There exists a constant \( c > 1 \) such that for all \( p \leq 0.999 \) with $p2^n\rightarrow \infty$, with high probability, \( h(G_{P_{n,p}}) \geq c \).
\end{conjecture}

Further exploration of this question may not only shed light on the probabilistic behavior of polytope graphs, but could also offer new insights into the general Mihail--Vazirani conjecture by contrasting worst-case and average-case behaviors.

\bibliography{literature}

\end{document}